\documentclass[11pt,a4paper]{article}
\usepackage[top=25mm,bottom=35mm,left=20mm,right=20mm,truedimen]{geometry}

\usepackage{amsmath,newtxtext,newtxmath,bm}

\usepackage{amsthm}
\usepackage{indentfirst}
\usepackage{mathtools}
\usepackage{algorithmic}
\usepackage{algorithm}
\usepackage{makecell}
\usepackage{multirow}
\usepackage{threeparttable}
\usepackage{booktabs}

\usepackage{color}

\newtheorem{defn}{\sffamily{Definition}}[section]

\newtheorem{proposition}[defn]{\sffamily{Proposition}}
\newtheorem{corollary}[defn]{\sffamily{Corollary}}

\theoremstyle{definition}
\newtheorem{example}{\sffamily{Example}}[section]

\numberwithin{algorithm}{section}
\numberwithin{equation}{section}

\allowdisplaybreaks

\title{
  Factorized Krylov subspace methods for solving large Sylvester equations
  }
\author{
    Yuki Satake\thanks{Information Initiative Center, Hokkaido University, Sapporo, Japan. Email: \texttt{\{satake, fukaya\}@iic.hokudai.ac.jp}}\and 
    Takeshi Fukaya$^\ast$, \and
    Tomohiro Sogabe\thanks{Graduate School of Engineering, Nagoya University, Nagoya, Japan. Email: \texttt{\{sogabe, zhang\}@na.nuap.nagoya-u.ac.jp}} \and
    Shao-Liang Zhang$^\dagger$ 
}
\date{}

\begin{document}
  \maketitle
  \begin{abstract}
    Krylov subspace methods, such as the Conjugate Gradient (CG) and BiCGSTAB methods, are widely used in scientific computing for solving linear systems. 
    In this study, we propose a new framework for solving large Sylvester equations in a low-rank format by reconstructing matrix-oriented Krylov subspace methods. 
    The framework realizes efficient algorithms that are mathematically equivalent to the matrix-oriented Krylov subspace methods by exploiting the mathematical properties of the Sylvester operator and the low-rank structure of the right-hand side. 
    Specifically, by leveraging these properties, approximate solutions can be expressed in a low-rank factorized form, enabling efficient computation and reduced memory requirements.
    The effectiveness of our algorithms is demonstrated through numerical experiments.
  \end{abstract}

  \section{Introduction}\label{sec:intro}
%% Target, Application
In this paper, we consider the Sylvester matrix equation
\begin{equation}
    AX + XB = C_1C_2^\top,\label{eq:sylvester}
\end{equation}
where $A\in\mathbb{R}^{n\times n}$, $B\in\mathbb{R}^{m\times m}$, $C_1\in\mathbb{R}^{n\times s}$, and $C_2\in\mathbb{R}^{m\times s}$ are given, and $X\in\mathbb{R}^{n\times m}$ is to be determined.
The Sylvester equation~\eqref{eq:sylvester} arises in many scientific fields, such as image restoration, control theory, and model reduction~\cite{calvetti1996,datta2004,sorensen2002}. 
The discretization of PDEs also yields Sylvester equations, e.g.,~\cite{palitta2016}.
When $B=A^\top$, Eq.~\eqref{eq:sylvester} is called the Lyapunov equation, which appears in various fields, including control theory and model reduction~\cite{antoulas2005,gajic2008}.

%& Existing studies (small ~ middle)
When the matrix sizes are small, the Bartels–Stewart algorithm~\cite{bartels1972} is commonly used to solve the Sylvester equation~\eqref{eq:sylvester}.
The algorithm transforms Eq.~\eqref{eq:sylvester} based on the Schur decompositions of $A$ and $B$, and computes the solution by performing backward substitutions on a sequence of triangular systems.
Other approaches have also been developed, including the Hessenberg–Schur method~\cite{golub1979} and a variant of the Bartels–Stewart algorithm~\cite{Sorensen2003}.

%% Existing studies (large)
When both $A$ and $B$ have large dimensions, memory consumption becomes a critical problem. 
Even if $A$ and $B$ are sparse, the solution matrix $X$ is generally dense, making it impractical to store all of its elements in very large-scale cases. 
However, when the right-hand side is a low-rank matrix, i.e., when $s\ll m,n$, the singular values of $X$ are expected to decay exponentially fast, and in this case, $X$ can be well-approximated by a low-rank matrix.
Such low-rank right-hand sides frequently arise in various application areas, including control theory and model reduction. 
For these cases, memory-efficient methods that compute low-rank approximate solutions—such as the ADI methods~\cite{benner2014,benner2009,benner2023} and projection methods~\cite{elguennouni2002,heyouni2010,hu1992,kressner2021,simoncini2007}—have been developed.
For more details of the existing methods, see~\cite{simoncini2016} and references therein.

%% Background
Another possible approach is to apply Krylov subspace methods—such as the conjugate gradient (CG)~\cite{hestenes1952} and BiCGSTAB~\cite{vandervorst1992} methods—to the Sylvester equation since it can be regarded as a certain type of linear system.
This can be observed by vectorizing~\eqref{eq:sylvester}, which yields the following equivalent linear system:
\begin{equation}
    \left(I_m\otimes A + B^\top \otimes I_n\right)\bm{x} = \bm{c}, \label{eq:ls}
\end{equation}
where $I_n$ is the $n\times n$ identity matrix, $\otimes$ denotes the Kronecker product, $\bm{x}:=\mathrm{vec}(X)\in\mathbb{R}^{mn}$, and $\bm{c}:=\mathrm{vec}\left(C_1C_2^\top\right)\in\mathbb{R}^{mn}$.
The vec operator, $\mathrm{vec}:\mathbb{R}^{n\times m}\to\mathbb{R}^{mn}$, converts a matrix into a column vector by stacking the columns one on top of each other (see~\cite[p. 63]{sogabe2022}).
The Krylov subspace methods designed for linear matrix equations, including the Sylvester equation, are referred to as matrix-oriented Krylov subspace methods.
These methods perform all vector operations that arise in the Krylov subspace methods for vectorized formulation, such as~\eqref{eq:ls}, in matrix form.
%% Problems
Unfortunately, standard matrix-oriented Krylov subspace methods require storing all elements of the solution matrix, which makes them impractical for large-scale problems. 
One possible remedy is to employ truncation techniques to store the approximate solution in a low-rank format rather than as a full dense matrix.
In this approach, small singular values are discarded at each iteration to maintain the low-rank structure of the approximate solution.
Matrix-oriented Krylov subspace methods with low-rank truncation (hereafter referred to as the truncated Krylov subspace methods) have been developed and shown to be effective, particularly for more general linear matrix equations such as multiterm matrix equations (see~\cite{benner2013,kressner2014,kressner2011,palitta2021,simoncini2023}). 
More recently, the subspace-conjugate gradient method, which is a further development of the truncated CG method, has been proposed in~\cite{palitta2025} and exploits richer subspace information by replacing the scalar coefficients in the CG method with small matrix coefficients. 
However, in these approaches, low-rank truncations are needed in each iteration, which can be computationally expensive.

%% Our study
Motivated by the aforementioned issues, we investigate theoretical aspects of the matrix-oriented Krylov subspace methods and redesign the algorithms to enhance their practical feasibility for large-scale problems.
We exploit the fact that each matrix arising in these methods can be represented in a low-rank factorized form.
This form is expressed as a factorization into the product of a column-orthonormal matrix, a small square matrix, and a row-orthonormal matrix, as in the projection methods.
Using such a factorized representation, fundamental operations implemented in the matrix-oriented Krylov subspace methods---such as the CG or BiCGSTAB method---can be computed efficiently.
We name the redesigned algorithms the \textit{factorized Krylov subspace methods}.
Importantly, unlike the existing projection methods, our approach does not require explicitly solving projected equations using direct solvers such as the Bartels-Stewart algorithm.
Instead, it leverages the recurrence relations inherent in the CG or BiCGSTAB method to update a small matrix at each iteration.
Although the rank of the approximate solution increases as the iterations proceed, a key advantage of our approach is that truncation is not required at every iteration.
Our numerical experiments confirm that this leads to a significant reduction in computational cost.
As in other existing methods for solving large matrix equations, we also assume that the right-hand side is low-rank, i.e., $s\ll m,n$.

%% Organization 
The remainder of this paper is organized as follows. Section~\ref{sec:existing} reviews the matrix-oriented Krylov subspace methods for the Sylvester equation. 
Section~\ref{sec:math} analyzes their mathematical structure and reformulates fundamental operations in a low-rank framework. Section~\ref{sec:main} introduces the proposed factorized Krylov subspace methods and details specific algorithms. 
Section~\ref{sec:experiments} reports numerical experiments demonstrating the efficiency of the proposed methods. 
Section~\ref{sec:conclusions} concludes the paper.

Throughout this paper, $O_{s_1,s_2}$ denotes the $s_1\times s_2$ zero matrix, and $\|\cdot\|$ denotes the Frobenius norm.
The notation $A(1:i,1:j)$ denotes the $i\times j$ upper-left submatrix of $A$.

%% ================================================ %%
\section{Matrix-oriented Krylov subspace methods}
\label{sec:existing}
% ------- Intro. ------ %
This section provides a brief review of matrix-oriented Krylov subspace methods for solving the Sylvester equation~\eqref{eq:sylvester}.
These methods can be naturally derived by representing all $mn$-dimensional vectors in the Krylov subspace methods for the equivalent linear system~\eqref{eq:ls} as $n\times m$ matrices.
Such a bijective correspondence between an $mn$-dimensional vector and an $n\times m$ matrix is established via the vec operator. 
Using the correspondence, the fundamental vector operations in Krylov subspace methods for the linear system~\eqref{eq:ls}---namely, vector addition, inner product, and matrix-vector product (Sylvester operator)---can be expressed in terms of matrix operations.
To see this, let us consider two matrices $X,Y\in\mathbb{R}^{n\times m}$, and corresponding vectors $\bm{x}:=\mathrm{vec}(X)\in\mathbb{R}^{mn}, \bm{y}:=\mathrm{vec}(Y)\in\mathbb{R}^{mn}$.
Then the following relations hold.

\begin{itemize}
    \item Addition: 
    the vector addition $\bm{x}+\bm{y}$ corresponds to the matrix addition $X+Y$ via the vec operator:
    \begin{equation*}
        \bm{x}+\bm{y}=\mathrm{vec}(X+Y)\in\mathbb{R}^{mn}.
    \end{equation*}
    \item Inner product:
    the inner product of $\bm{x}$ and $\bm{y}$ is equal to the inner product of $X$ and $Y$ as defined below:
    \begin{equation*}
        \left\langle\bm{x}, \bm{y}\right\rangle:=\bm{x}^\top \bm{y}=\mathrm{tr}(X^\top Y)=:\left\langle X, Y\right\rangle.
    \end{equation*}
    \item Sylvester operator:
    the matrix-vector product of $\mathcal{F}:=I_m\otimes A + B^\top \otimes I_n\in\mathbb{R}^{mn\times mn}$ and $\bm{x}$ represents a vector form of the Sylvester operator, i.e.,
    \begin{equation*}
        \mathcal{F}\bm{x} = \mathrm{vec}(AX+XB).
    \end{equation*}
\end{itemize}
Taking the above correspondence into account, one can easily derive matrix-oriented Krylov subspace methods.
Algorithms~\ref{alg:CG} and~\ref{alg:BiCGSTAB} are examples of the matrix-oriented Krylov subspace methods for solving the Sylvester equation~\eqref{eq:sylvester}, namely the matrix-oriented CG method and the matrix-oriented BiCGSTAB method, respectively.

\begin{algorithm}
\caption{Matrix-oriented CG method for solving~\eqref{eq:sylvester}}
\label{alg:CG}
\begin{algorithmic}[1]
\REQUIRE $A\in\mathbb{R}^{n\times n},B\in\mathbb{R}^{m\times m},C(:=C_1C_2^\top)\in\mathbb{R}^{n\times m},\varepsilon_\mathrm{tol}$
\ENSURE $X_{k+1}\in\mathbb{R}^{n\times m}$
\STATE{Set initial guess $X_0$}
\STATE{Compute $R_0=C-AX_0-X_0 B$}
\STATE{$P_0=R_0$, $\rho_0=\langle R_0, R_0\rangle$}
\FOR{$k=0,1,2,\ldots$}
\STATE{$Q_{k} = AP_{k}+P_{k}B$}
\STATE{$\alpha_k=\rho_k/\langle P_k, Q_k \rangle$}
\STATE{$X_{k+1}=X_k + \alpha_k P_k$}
\STATE{$R_{k+1}=R_k - \alpha_k Q_k$}
\IF{$\|R_{k+1}\|\le \|R_0\|\cdot \varepsilon_\mathrm{tol}$}
\RETURN $X_{k+1}$
\ENDIF
\STATE{$\rho_{k+1}=\langle R_{k+1}, R_{k+1}\rangle$}
\STATE{$\beta_k=\rho_{k+1}/\rho_k$}
\STATE{$P_{k+1}=R_{k+1} + \beta_k P_k$}
\ENDFOR
% \RETURN $X_{k+1}$
\end{algorithmic}
\end{algorithm}

\begin{algorithm}
\caption{Matrix-oriented BiCGSTAB method for solving~\eqref{eq:sylvester}}
\label{alg:BiCGSTAB}
\begin{algorithmic}[1]
\REQUIRE $A\in\mathbb{R}^{n\times n},B\in\mathbb{R}^{m\times m},C(:=C_1C_2^\top)\in\mathbb{R}^{n\times m},\varepsilon_\mathrm{tol}$
\ENSURE {$X_{k+1}\in\mathbb{R}^{n\times m}$}
\STATE{Set initial guess $X_0$}
\STATE{Compute $R_0=C-AX_0-X_0 B$}
\STATE{Set an arbitrary matrix $\tilde{R}_0$ s.t. $\langle R_0,\tilde{R}_0\rangle\ne0$, e.g., $\tilde{R}_0=R_0$}
\STATE{$P_0=R_0$, $\rho_0=\langle R_0, \tilde{R}_0\rangle$}
\FOR{$k=0,1,2,\ldots$}
\STATE{$Q_{k} = AP_{k}+P_{k}B$}
\STATE{$\alpha_k=\rho_k/\langle \tilde{R}_0, Q_k \rangle$}
\STATE{$S_k=R_k - \alpha_kQ_k$}
\STATE{$T_{k} = AS_{k}+S_{k}B$}
\STATE{$\omega_k = \langle T_k,S_k\rangle/\langle T_k,T_k\rangle$}
\STATE{$X_{k+1}=X_k + \alpha_k P_k + \omega_k S_k$}
\STATE{$R_{k+1}=S_k - \omega_k T_k$}
\IF{$\|R_{k+1}\|\le \|R_0\|\cdot \varepsilon_\mathrm{tol}$}
\RETURN $X_{k+1}$
\ENDIF
\STATE{$\rho_{k+1}=\langle R_{k+1}, \tilde{R}_0\rangle$}
\STATE{$\beta_k=\alpha_k/\omega_k \cdot\rho_{k+1}/\rho_k$}
\STATE{$P_{k+1}=R_{k+1} + \beta_k (P_k - \omega_k Q_k)$}
\ENDFOR
% \RETURN $X_{k+1}$
\end{algorithmic}
\end{algorithm}

%% ================================================ %%
\section{Factorized representation of the Krylov subspace}\label{sec:math}
The original matrix-oriented Krylov subspace methods for~\eqref{eq:sylvester} involve operations with $n\times m$ dense matrices, which may lead to issues with computational cost and memory requirements for large cases.
To address such issues, we first discuss mathematical properties of the Krylov subspace 
\begin{equation}
    \mathcal{K}_k(\mathcal{F},\bm{c}):=\mathrm{span}\left(\bm{c},\mathcal{F}\bm{c},\mathcal{F}^2\bm{c},\ldots,\mathcal{F}^{k-1}\bm{c}\right)\label{v-KS}
\end{equation}
that underlie the matrix-oriented Krylov subspace methods for~\eqref{eq:sylvester}.
We then exploit the mathematical properties to derive a low-rank factorized form, and use it to reformulate the basic operations in the matrix-oriented Krylov subspace method.

\subsection{Mathematical discussion}\label{subsec:3-1}
For the Krylov subspace~\eqref{v-KS}, the following property holds:
\begin{proposition}\label{prop:sub}
    Let $\mathcal{F}:=I_m\otimes A + B^\top \otimes I_n\in\mathbb{R}^{mn\times mn}$ and $\bm{c}:=\mathrm{vec}(C_1C_2^\top)\in\mathbb{R}^{mn}$, where $A\in\mathbb{R}^{n\times n}, B\in\mathbb{R}^{m\times m}$, $C_1\in\mathbb{R}^{n\times s}$, and $C_2\in\mathbb{R}^{m\times s}$.
    Then, 
    \begin{equation}
        \mathcal{K}_k(\mathcal{F},\bm{c})\subseteq \mathcal{K}_k^\square(B^\top,C_2)\otimes \mathcal{K}_k^\square(A,C_1),
    \end{equation}
    where 
    $\mathcal{K}_k^\square(A,C_1)$ denotes the range of the matrix $[C_1,AC_1,A^2C_1,\ldots, A^{k-1}C_1]$ and $\mathcal{K}_k^\square(B^\top,C_2)$ denotes the range of the matrix $[C_2,B^\top C_2,(B^\top)^2C_2,\ldots, (B^\top)^{k-1}C_2]$.
    % \begin{align*}
    %     \mathcal{K}_k^\square(A,C_1)&:=\mathrm{range}\left([C_1,AC_1,A^2C_1,\ldots, A^{k-1}C_1]\right), \\
    %     \mathcal{K}_k^\square(B^\top,C_2)&:=\mathrm{range}\left([C_2,B^\top C_2,(B^\top)^2C_2,\ldots, (B^\top)^{k-1}C_2]\right).
    % \end{align*}
    Here, the Kronecker product of the subspaces $\mathcal{K}_k^\square(A,C_1)$ and $\mathcal{K}_k^\square(B^\top,C_2)$ is defined as 
    \begin{equation*}
        \mathcal{K}_k^\square(B^\top,C_2)\otimes \mathcal{K}_k^\square(A,C_1):=\mathrm{span}\{\bm{w}\otimes\bm{v} : \bm{v}\in\mathcal{K}_k^\square(A,C_1), \bm{w}\in\mathcal{K}_k^\square(B^\top,C_2)\}.
    \end{equation*}
\end{proposition}
\begin{proof}
    From the definition~\eqref{v-KS}, an arbitrary vector $\bm{z}\in\mathcal{K}_k(\mathcal{F},\bm{c})$ can be expressed as a linear combination of basis vectors $\bm{c},\mathcal{F}\bm{c},\ldots,\mathcal{F}^{k-1}\bm{c}$, i.e., 
    \begin{equation}
        \bm{z}=\sum^{k-1}_{i=0}\rho_i\mathcal{F}^{i}\bm{c},
    \end{equation}
    where $\rho_i\in\mathbb{R}$. 
    From the structure of $\mathcal{F}=I_m\otimes A + B^\top \otimes I_n$, it follows that
    \begin{align*}
        \mathcal{F}^i\bm{c}&=\left(I_m\otimes A + B^\top \otimes I_n\right)^i\bm{c}\\
        &=\sum^{i}_{l=0}\begin{pmatrix}
            i \\ l
        \end{pmatrix}\left[\left(B^\top\right)^{i-l}\otimes A^{l}\right]\bm{c},
    \end{align*}
    where $\begin{pmatrix}i \\ l\end{pmatrix}$ denotes the binomial coefficient.
    Then, $\bm{z}\in\mathcal{K}_k(\mathcal{F},\bm{c})$ can be expressed as
    \begin{equation*}
        \bm{z}=\sum^{k-1}_{i=0}\left\{\rho_i\sum^{i}_{l=0}\begin{pmatrix}
            i \\ l
        \end{pmatrix}\left[\left(B^\top\right)^{i-l}\otimes A^{l}\right]\bm{c}\right\}.
    \end{equation*}
    Since $\bm{c}=\mathrm{vec}(C_1C_2^\top)=(C_2\otimes C_1)\mathrm{vec}(I_s)$, it follows that
    \begin{equation}
        \bm{z}=\sum^{k-1}_{i=0}\left\{\rho_i\sum^{i}_{l=0}\begin{pmatrix}
            i \\ l
        \end{pmatrix}\left[\left(\left(B^\top\right)^{i-l}C_2\right) \otimes \left(A^{l}C_1\right)\right]\mathrm{vec}(I_s)\right\}.\label{eq:z_in_K}
    \end{equation}
    On the other hand, for any $\tilde{\bm{z}}\in\mathcal{K}_k^\square(B^\top,C_2)\otimes\mathcal{K}_k^\square(A,C_1)$, there exist matrices $\hat{Z}_{ij}\in\mathbb{R}^{s\times s}$ such that
    \begin{equation}
        \tilde{\bm{z}}
        = \sum^{k-1}_{i=0}\sum^{k-1}_{j=0} \left\{\left[(B^\top)^jC_2\right]\otimes(A^iC_1)\right\} \mathrm{vec}(\hat{Z}_{ij}).\label{eq:z_in_KK}
    \end{equation}
    Hence, it is obvious that~\eqref{eq:z_in_K} is of the form~\eqref{eq:z_in_KK}, which completes the proof.
\end{proof}

When $C_1$ and $C_2$ are vectors, i.e., $C_1\in\mathbb{R}^n, C_2\in\mathbb{R}^m$, the above proposition corresponds to a special case of the statement in~\cite[Section 3.1]{kressner2010}.
Proposition~\ref{prop:sub} shows that the Krylov subspace $\mathcal{K}_k(\mathcal{F},\bm{c})$ can be factorized as the (Kronecker) product of two block Krylov subspaces $\mathcal{K}_k^\square(A,C_1)$ and $\mathcal{K}_k^\square(B^\top,C_2)$.
This fact leads to the following result, obtained by applying the inverse vec operator to~\eqref{eq:z_in_K}.
\begin{corollary}\label{cor:z}
    Let $\bm{z}=\mathrm{vec}(Z)\in\mathcal{K}_k(\mathcal{F},\bm{c})$, where $\mathcal{F}$ and $\bm{c}$ are as defined in Proposition~\ref{prop:sub}, and $Z\in\mathbb{R}^{n\times m}$.
    Then,
    \begin{equation}\label{eq:lr-express}
        Z = \begin{bmatrix}C_1, AC_1, \ldots ,A^{k-1}C_1\end{bmatrix}\hat{\mathbf{Z}}\begin{bmatrix}C_2, B^\top C_2, \ldots ,\left(B^\top\right)^{k-1}C_2\end{bmatrix}^\top,
    \end{equation}
     where $\hat{\mathbf{Z}}\in\mathbb{R}^{ks\times ks}$ is a block upper anti-triangular matrix with the $(i,j)$-th block matrix $\hat{Z}_{ij}\in\mathbb{R}^{s\times s}$ defined by
    \begin{equation*}
        \hat{Z}_{ij} = \begin{cases}
            \begin{pmatrix}
                i+j-2 \\ i-1
            \end{pmatrix}\rho_{i+j-2}I_s & \text{if $i+j \le k+1$,} \\
            O_{s,s}       & \text{if $i+j > k+1$.}
        \end{cases} 
    \end{equation*}
\end{corollary}

The above result shows that all matrices $Z\in\mathbb{R}^{n\times m}$ such that $\mathrm{vec}(Z)\in\mathcal{K}_k(\mathcal{F},\bm{c})$ are of rank at most $ks$, which implies that $Z$ is low-rank when $ks\ll m,n$.
This motivates us to redesign the matrix-oriented Krylov subspace methods for~\eqref{eq:sylvester} based on the low-rank factorized form~\eqref{eq:lr-express}.
Such a low-rank representation avoids storing large dense matrices, which is expected to reduce memory consumption and improve computational efficiency.
To this end, we introduce the following corollary, which provides a more tractable low-rank factorized form suited to the reformulation of the basic operations discussed later.
\begin{corollary}\label{cor:z2}
The matrix $Z$ in Corollary~\ref{cor:z} can be represented as follows:
\begin{equation}
    Z=\mathcal{V}_k\mathbf{Z}\mathcal{W}_k^\top,\label{form:VZW}
    % \red{Z=\mathcal{V}_k\mathcal{Z}\mathcal{W}_k^\top,\label{form:VZW}}
\end{equation}
where the columns of $\mathcal{V}_k\in\mathbb{R}^{n\times ks}$ and $\mathcal{W}_k\in\mathbb{R}^{m\times ks}$ form orthonormal bases of the block Krylov subspaces $\mathcal{K}_k^\square(A,C_1)$ and $\mathcal{K}_k^\square(B^\top,C_2)$, respectively. 
\end{corollary}

\subsection{Reformulation of basic operations}
We reformulate the basic operations of the matrix-oriented Krylov subspace methods described in Section~\ref{sec:existing} by employing the expression~\eqref{form:VZW}.
Note that, throughout this paper, bold upright notation is used to denote the corresponding small matrix in a low-rank factorized form (e.g., $Z=\mathcal{V}_k\mathbf{Z}\mathcal{W}_k^\top$), in order to distinguish it from the original large matrix.

To utilize the low-rank factorized form~\eqref{form:VZW}, it is required to obtain orthonormal bases of the block Krylov subspaces $\mathcal{K}_k^\square(A,C_1)$ and $\mathcal{K}^\square_k(B^\top,C_2)$, respectively.
Such bases can be computed by the block Arnoldi process~\cite{saad2011} (or the block Lanczos process for the symmetric case~\cite{cullum1974,golub1977}).
Hereafter, $\mathcal{V}_k$ and $\mathcal{W}_k$ are assumed to be computed by the block Arnoldi process or the block Lanczos process.
Therefore, the following relations hold:
\begin{equation}
    A\mathcal{V}_k = \mathcal{V}_{k+1}\mathcal{H}_{k+1,k}, \quad B^\top\mathcal{W}_k = \mathcal{W}_{k+1}\mathcal{G}_{k+1,k},\label{eq:BK}
\end{equation}
where $\mathcal{H}_{k+1,k}\in\mathbb{R}^{(k+1)s\times ks}, \mathcal{G}_{k+1,k}\in\mathbb{R}^{(k+1)s\times ks}$ are block Hessenberg matrices, and $\mathcal{V}_k:=\left[V_1,V_2,\ldots,V_k\right]\in\mathbb{R}^{n\times ks},\mathcal{W}_k:=\left[W_1,W_2,\ldots,W_k\right]\in\mathbb{R}^{m\times ks}$ are column-orthonormal matrices.
In what follows, we assume $m,n$ to be much larger than $k,s$, so that $ks\ll m,n$.

Let $X,Y\in\mathbb{R}^{n\times m}$ satisfy $\mathrm{vec}(X),\mathrm{vec}(Y)\in\mathcal{K}_k(\mathcal{F},\bm{c})$.
In this case, from Corollary~\ref{cor:z2}, there exist $\mathbf{X},\mathbf{Y}\in\mathbb{R}^{ks\times ks}$ such that
\begin{equation*}
    X=\mathcal{V}_k\mathbf{X}\mathcal{W}_k^\top,\quad Y=\mathcal{V}_k\mathbf{Y}\mathcal{W}_k^\top.
    % \red{X=\mathcal{V}_k\mathcal{X}\mathcal{W}_k^\top,\quad Y=\mathcal{V}_k\mathcal{Y}\mathcal{W}_k^\top.}
\end{equation*}
Using the above expression, each operation can be reformulated as follows.
\begin{itemize}
    \item Addition: Matrix addition $X+Y$ can be represented as
    \begin{equation}
        X+Y = \mathcal{V}_k(\mathbf{X} + \mathbf{Y})\mathcal{W}_k^\top,\label{oper:addition}
        % \red{X+Y = \mathcal{V}_k(\mathcal{X} + \mathcal{Y})\mathcal{W}_k^\top,\label{oper:addition}}
    \end{equation}
    which indicates that the addition of $n\times m$ matrices can be reduced to the addition of $ks\times ks$ matrices.
    Then, the computational cost decreases from $\mathcal{O}(mn)$ to $\mathcal{O}(k^2s^2)$.
    \item Inner product: The inner product of $X$ and $Y$ can be rewritten as
    \begin{equation}
        \langle X,Y\rangle=\mathrm{tr}(X^\top Y)=\mathrm{tr}(\mathcal{W}_k\mathbf{X}^\top\mathcal{V}_k^\top \mathcal{V}_k\mathbf{Y}\mathcal{W}_k^\top)=\langle\mathbf{X},\mathbf{Y}\rangle.\label{oper:inner_product}
        % \red{(X,Y)=\mathrm{tr}(X^\top Y)=\mathrm{tr}(\mathcal{W}_k\mathcal{X}^\top\mathcal{V}_k^\top \mathcal{V}_k\mathcal{Y}\mathcal{W}_k^\top)=(\mathcal{X},\mathcal{Y}).\label{oper:inner_product}}
    \end{equation}
    This implies that the inner product of $n\times m$ matrices can be reduced to the inner product of $ks\times ks$ matrices, thereby reducing the computational cost from $\mathcal{O}(mn)$ to $\mathcal{O}(k^2s^2)$.
    \item Sylvester operator: From the relation~\eqref{eq:BK}, it follows that
    \begin{align}
        AX+XB &= A\mathcal{V}_k\mathbf{X}\mathcal{W}_k^\top + \mathcal{V}_k\mathbf{X}\mathcal{W}_k^\top B \nonumber \\
        & = \mathcal{V}_{k+1}\left(\left[\mathcal{H}_{k+1,k}\mathbf{X}, O_{ks,s}\right] + \begin{bmatrix}\mathbf{X}\mathcal{G}_{k+1,k}^\top \\ O_{s,ks}\end{bmatrix}\right)\mathcal{W}_{k+1}^\top.\label{oper:sylvester-operator}
        % AX+XB &= A\mathcal{V}_k\mathcal{X}\mathcal{W}_k^\top + \mathcal{V}_k\mathcal{X}\mathcal{W}_k^\top B \nonumber \\
        % & = \mathcal{V}_{k+1}\left(\left[\mathcal{H}_{k+1,k}\mathcal{X}, O_{ks,s}\right] + \begin{bmatrix}\mathcal{X}\mathcal{G}_{k+1,k}^\top \\ O_{s,ks}\end{bmatrix}\right)\mathcal{W}_{k+1}^\top.\label{oper:sylvester-operator}
    \end{align}
    This allows the multiplication between a large sparse matrix and a large dense matrix to be reduced to the multiplication between small dense matrices, which reduces the computational cost from $\mathcal{O}(\mathrm{nnz}(A)\cdot m + \mathrm{nnz}(B)\cdot n + mn)$ to $\mathcal{O}(k^3s^3)$.
    Although the Sylvester operator in the low-rank factorized form requires additional computations for $\mathcal{V}_{k+1}$ and $\mathcal{W}_{k+1}$, 
    these matrices can be obtained by performing the $k$-th step of the block Arnoldi process (or block Lanczos process) for $\mathcal{K}_k^\square(A,C_1)$ and $\mathcal{K}_k^\square(B^\top,C_2)$, respectively.
    The computational cost of the $k$-th step of the block Arnoldi process is $\mathcal{O}\left((\mathrm{nnz}(A)+\mathrm{nnz}(B))s + k(m+n)s^2\right)$, while that of the block Lanczos process is $\mathcal{O}\left((\mathrm{nnz}(A)+\mathrm{nnz}(B))s + (m+n)s^2\right)$.
\end{itemize}

\section{Factorized Krylov subspace methods for the Sylvester equation}
\label{sec:main}
In this section, we redesign the matrix-oriented Krylov subspace methods by using the low-rank factorized form discussed in the previous sections.
This paper describes only the CG and BiCGSTAB methods; however, our approach can be applied to other Krylov subspace methods, such as the CR~\cite{stiefel1955}, CGS~\cite{sonneveld1989}, and GPBiCG~\cite{zhang1997} methods.

\subsection{Symmetric case}
We first consider reconstructing the matrix-oriented CG method for the Sylvester equation~\eqref{eq:sylvester}.
In this subsection, we assume that $A$ and $B$ are symmetric positive definite.
In view of the correspondence with the equivalent linear system~\eqref{eq:ls}, the following relations hold for the matrices in Algorithm~\ref{alg:CG}:
\begin{align*}
    \mathrm{vec}(X_k-X_0)\in\mathcal{K}_{k}(\mathcal{F},\mathrm{vec}(R_0)), \quad \mathrm{vec}(R_k)\in\mathcal{K}_{k+1}(\mathcal{F},\mathrm{vec}(R_0)), \\
    \mathrm{vec}(P_k)\in\mathcal{K}_{k+1}(\mathcal{F},\mathrm{vec}(R_0)), \quad \mathrm{vec}(Q_k)\in\mathcal{K}_{k+2}(\mathcal{F},\mathrm{vec}(R_0)).
\end{align*}
Here, to represent the approximate solutions in the low-rank factorized form~\eqref{form:VZW}, we set $X_0 = O_{n,m}$, which implies that 
\begin{align*}
    \mathrm{vec}(X_k)\in\mathcal{K}_{k}(\mathcal{F},\bm{c}), \quad \mathrm{vec}(R_k)\in\mathcal{K}_{k+1}(\mathcal{F},\bm{c}), \\
    \mathrm{vec}(P_k)\in\mathcal{K}_{k+1}(\mathcal{F},\bm{c}), \quad \mathrm{vec}(Q_k)\in\mathcal{K}_{k+2}(\mathcal{F},\bm{c}).
\end{align*}
Consequently, each matrix can be represented as the following low-rank factorized form:
\begin{align*}
    X_k=\mathcal{V}_{k}\mathbf{X}_k\mathcal{W}_{k}^\top, \quad 
    R_k = \mathcal{V}_{k+1}\mathbf{R}_k\mathcal{W}_{k+1}^\top, \\
    P_k = \mathcal{V}_{k+1}\mathbf{P}_k\mathcal{W}_{k+1}^\top, \quad 
    Q_k = \mathcal{V}_{k+2}\mathbf{Q}_k\mathcal{W}_{k+2}^\top, 
\end{align*}
where
\begin{align*}
    \mathbf{X}_k\in\mathbb{R}^{ks \times ks}, \quad
    \mathbf{R}_k\in\mathbb{R}^{(k+1)s \times (k+1)s}, \\
    \mathbf{P}_k\in\mathbb{R}^{(k+1)s \times (k+1)s}, \quad
    \mathbf{Q}_k\in\mathbb{R}^{(k+2)s \times (k+2)s}.
\end{align*}
Since $A$ and $B$ are symmetric, the block Lanczos process can be used to compute $\mathcal{V}_k$ and $\mathcal{W}_k$.
Using~\eqref{oper:addition}, the lines 7 and 8 of Algorithm~\ref{alg:CG} can be written as
\begin{align*}
    X_{k+1} &= X_k + \alpha_k P_k = \mathcal{V}_{k+1}\left(\begin{bmatrix}\mathbf{X}_k & O_{ks,s} \\ O_{s,ks} & O_{s,s}\end{bmatrix} + \alpha_k\mathbf{P}_k\right)\mathcal{W}_{k+1}^\top, \\
    R_{k+1} &= R_k - \alpha_k Q_k = \mathcal{V}_{k+2}\left(\begin{bmatrix}\mathbf{R}_k & O_{(k+1)s,s} \\ O_{s,(k+1)s} & O_{s,s}\end{bmatrix} - \alpha_k\mathbf{Q}_k\right)\mathcal{W}_{k+2}^\top.
    % X_{k+1} &= X_k + \alpha_k P_k = \mathcal{V}_{k+1}\left(\begin{bmatrix}\mathcal{X}_k & O_{ks,s} \\ O_{s,ks} & O_{s,s}\end{bmatrix} + \alpha_k\mathcal{P}_k\right)\mathcal{W}_{k+1}^\top, \\
    % R_{k+1} &= R_k - \alpha_k Q_k = \mathcal{V}_{k+2}\left(\begin{bmatrix}\mathcal{R}_k & O_{(k+1)s,s} \\ O_{s,(k+1)s} & O_{s,s}\end{bmatrix} - \alpha_k\mathcal{Q}_k\right)\mathcal{W}_{k+2}^\top.
\end{align*}

Therefore, we can update the small matrices $\mathbf{X}_k$ and $\mathbf{R}_k$ by the following expressions
\begin{align*}
    \mathbf{X}_{k+1}&:=\begin{bmatrix}\mathbf{X}_k & O_{ks,s} \\ O_{s,ks} & O_{s,s}\end{bmatrix} + \alpha_k\mathbf{P}_k, \\ % \in\mathbb{R}^{(k+1)s\times (k+1)s}
    \mathbf{R}_{k+1}&:=\begin{bmatrix}\mathbf{R}_k & O_{(k+1)s,s} \\ O_{s,(k+1)s} & O_{s,s}\end{bmatrix} - \alpha_k\mathbf{Q}_k,  % \in\mathbb{R}^{(k+2)s\times (k+2)s}
    % \mathcal{X}_{k+1}&:=\begin{bmatrix}\mathcal{X}_k & O_{ks,s} \\ O_{s,ks} & O_{s,s}\end{bmatrix} + \alpha_k\mathcal{P}_k, \\ % \in\mathbb{R}^{(k+1)s\times (k+1)s}
    % \mathcal{R}_{k+1}&:=\begin{bmatrix}\mathcal{R}_k & O_{(k+1)s,s} \\ O_{s,(k+1)s} & O_{s,s}\end{bmatrix} - \alpha_k\mathcal{Q}_k,  % \in\mathbb{R}^{(k+2)s\times (k+2)s}
\end{align*}
instead of the $n\times m$ dense matrices $X_k$ and $R_k$.
Similarly, the line 14 of Algorithm~\ref{alg:CG} can be written as
\begin{equation*}
    P_{k+1} = R_{k+1} + \beta_kP_k = \mathcal{V}_{k+2}\left( \mathbf{R}_{k+1} + \begin{bmatrix}\beta_k\mathbf{P}_k & O_{(k+1)s,s} \\ O_{s,(k+1)s} & O_{s,s}\end{bmatrix} \right)\mathcal{W}_{k+2}^\top.
    % P_{k+1} = R_{k+1} + \beta_kP_k = \mathcal{V}_{k+2}\left( \mathcal{R}_{k+1} + \begin{bmatrix}\beta_k\mathcal{P}_k & O_{(k+1)s,s} \\ O_{s,(k+1)s} & O_{s,s}\end{bmatrix} \right)\mathcal{W}_{k+2}^\top.
\end{equation*}
Then we obtain the following update formula:
\begin{equation*}
\mathbf{P}_{k+1}:=\mathbf{R}_{k+1} + \begin{bmatrix}\beta_k\mathbf{P}_k & O_{(k+1)s,s} \\ O_{s,(k+1)s} & O_{s,s}\end{bmatrix}.  % \in\mathbb{R}^{(k+2)s\times (k+2)s}}
% \red{\mathcal{P}_{k+1}:=\mathcal{R}_{k+1} + \begin{bmatrix}\beta_k\mathcal{P}_k & O_{(k+1)s,s} \\ O_{s,(k+1)s} & O_{s,s}\end{bmatrix}.}  % \in\mathbb{R}^{(k+2)s\times (k+2)s}
\end{equation*}

With respect to the inner products in the lines 6 and 12 of Algorithm~\ref{alg:CG}, the expression~\eqref{oper:inner_product} gives
\begin{equation*}
    \alpha_k = \rho_k/\left\langle \mathbf{P}_k, \mathbf{Q}_k(1:(k+1)s,1:(k+1)s) \right\rangle, \quad \rho_{k+1} = \langle \mathbf{R}_{k+1}, \mathbf{R}_{k+1} \rangle.
    % \red{\rho_k = \langle \mathcal{R}_k, \mathcal{R}_k \rangle, \quad \alpha_k = \rho_k/\left\langle \mathcal{R}_k, \mathcal{Q}_k(1:(k+1)s,1:(k+1)s) \right\rangle.}
\end{equation*}

The calculation of the Sylvester operator in the line 5 of Algorithm~\ref{alg:CG} becomes
\begin{equation*}
    Q_{k} = AP_{k} + P_{k}B = \mathcal{V}_{k+2}\mathbf{Q}_{k}\mathcal{W}_{k+2}^\top,
\end{equation*}
where
\begin{equation*}
    \mathbf{Q}_{k}:=\left[\mathcal{H}_{k+2,k+1}\mathbf{P}_{k}, O_{(k+2)s,s}\right] + \begin{bmatrix}\mathbf{P}_{k}\mathcal{G}_{k+2,k+1}^\top \\ O_{s,(k+2)s}\end{bmatrix}  % \in\mathbb{R}^{(k+2)s\times (k+2)s}
\end{equation*}
by using~\eqref{oper:sylvester-operator}.

The convergence criterion is that the relative residual norm in Algorithm~\ref{alg:CG} is less than the threshold $\varepsilon_\mathrm{tol}$, that is $\|R_{k+1}\|/\|R_0\|<\varepsilon_\mathrm{tol}$, which can be rewritten as
\begin{equation*}
    \|R_{k+1}\|/\|R_0\| = \|\mathbf{R}_{k+1}\|/\|\mathbf{R}_0\|<\varepsilon_\mathrm{tol}.
\end{equation*}

Consolidating the above discussion, the matrix-oriented CG method for~\eqref{eq:sylvester} can be reconstructed with low-rank format. 
We call the reconstructed algorithm the {\it factorized CG method} and present it in Algorithm~\ref{alg:f-CG}.

\begin{algorithm}
\caption{Factorized CG method for solving~\eqref{eq:sylvester}}
\label{alg:f-CG}
\begin{algorithmic}[1]
\REQUIRE $A\in\mathbb{R}^{n\times n},B\in\mathbb{R}^{m\times m},C_1\in\mathbb{R}^{n\times s}, C_2\in\mathbb{R}^{m\times s},\varepsilon_\mathrm{tol}$
\ENSURE $\mathcal{V}_{k+1}\in\mathbb{R}^{n\times (k+1)s},\mathbf{X}_{k+1}\in\mathbb{R}^{(k+1)s\times (k+1)s},\mathcal{W}_{k+1}\in\mathbb{R}^{m\times (k+1)s}$ such that $\mathcal{V}_{k+1}\mathbf{X}_{k+1}\mathcal{W}_{k+1}^\top$ is an approximate solution to~\eqref{eq:sylvester}
% \STATE{Set initial guess $\mathbf{X}_0=0$}
\STATE{Compute QR factorizations: $C_1=V_1 R_\mathrm{L}, C_2=W_1R_\mathrm{R}$}
\STATE{$\mathbf{X}_0=\left[~\right]$, $\mathbf{R}_0=R_\mathrm{L} R_\mathrm{R}^\top$, $\mathbf{P}_0=\mathbf{R}_0$, $\rho_0=\langle \mathbf{R}_0, \mathbf{R}_0\rangle$}
\FOR{$k=0,1,2,\ldots$}
\STATE{Compute the $(k+1)$-th step of the block Lanczos process for $\mathcal{K}_{k+1}^\square(A,C_1)$ and obtain $\mathcal{V}_{k+2}$ and $\mathcal{H}_{k+2,k+1}$}
\STATE{Compute the $(k+1)$-th step of the block Lanczos process for $\mathcal{K}_{k+1}^\square(B^\top,C_2)$ and obtain $\mathcal{W}_{k+2}$ and $\mathcal{G}_{k+2,k+1}$}
\STATE{$\mathbf{Q}_k = \left[\mathcal{H}_{k+2,k+1}\mathbf{P}_{k}, O_{(k+2)s,s}\right] + \begin{bmatrix}\mathbf{P}_k\mathcal{G}_{k+2,k+1}^\top \\ O_{s,(k+2)s}\end{bmatrix}$}
\STATE{$\alpha_k=\rho_k/\left\langle \mathbf{P}_k, \mathbf{Q}_k(1:(k+1)s,1:(k+1)s) \right\rangle$}
\STATE{$\mathbf{X}_{k+1} = \begin{bmatrix}\mathbf{X}_k & O_{ks,s} \\ O_{s,ks} & O_{s,s}\end{bmatrix} + \alpha_k\mathbf{P}_k$}
\STATE{$\mathbf{R}_{k+1} = \begin{bmatrix}\mathbf{R}_k & O_{(k+1)s,s} \\ O_{s,(k+1)s} & O_{s,s}\end{bmatrix} - \alpha_k\mathbf{Q}_k$}
\IF{$\|\mathbf{R}_{k+1}\|\le \|\mathbf{R}_0\|\cdot \varepsilon_\mathrm{tol}$}
\RETURN $\mathbf{X}_{k+1}$, $\mathcal{V}_{k+1}$, and $\mathcal{W}_{k+1}$
\ENDIF
\STATE{$\rho_{k+1}=\langle \mathbf{R}_{k+1}, \mathbf{R}_{k+1}\rangle$}
\STATE{$\beta_k=\rho_{k+1}/\rho_k$}
\STATE{$\mathbf{P}_{k+1}=\mathbf{R}_{k+1} +  \begin{bmatrix}\beta_k\mathbf{P}_k & O_{(k+1)s,s} \\ O_{s,(k+1)s} & O_{s,s}\end{bmatrix}$}
\ENDFOR
\end{algorithmic}
\end{algorithm}

\subsection{Nonsymmetric case}
This subsection is devoted to reconstructing the matrix-oriented BiCGSTAB method for~\eqref{eq:sylvester} with nonsymmetric $A$ and $B$.
As with the CG method in the previous subsection, let initial guess $X_0$ be a zero matrix $O_{n, m}$.
Then, it follows that
\begin{align*}
    \mathrm{vec}(X_k)\in\mathcal{K}_{2k}(\mathcal{F},\bm{c}), \quad \mathrm{vec}(R_k)\in\mathcal{K}_{2k+1}(\mathcal{F},\bm{c}), \\
    \mathrm{vec}(P_k)\in\mathcal{K}_{2k+1}(\mathcal{F},\bm{c}), \quad \mathrm{vec}(Q_k)\in\mathcal{K}_{2k+2}(\mathcal{F},\bm{c}), \\
    \mathrm{vec}(S_k)\in\mathcal{K}_{2k+2}(\mathcal{F},\bm{c}), \quad \mathrm{vec}(T_k)\in\mathcal{K}_{2k+3}(\mathcal{F},\bm{c}),
\end{align*}
for the matrices in Algorithm~\ref{alg:BiCGSTAB}.
From the above and using the low-rank factorized form~\eqref{form:VZW}, each matrix can be represented as follows:
\begin{align*}
    X_k=\mathcal{V}_{2k}\mathbf{X}_k\mathcal{W}_{2k}^\top, \quad 
    R_k = \mathcal{V}_{2k+1}\mathbf{R}_k\mathcal{W}_{2k+1}^\top, \\
    P_k = \mathcal{V}_{2k+1}\mathbf{P}_k\mathcal{W}_{2k+1}^\top, \quad 
    Q_k = \mathcal{V}_{2k+2}\mathbf{Q}_k\mathcal{W}_{2k+2}^\top, \\
    S_k = \mathcal{V}_{2k+2}\mathbf{S}_k\mathcal{W}_{2k+2}^\top, \quad
    T_k = \mathcal{V}_{2k+3}\mathbf{T}_k\mathcal{W}_{2k+3}^\top,
\end{align*}
where
\begin{align*}
    \mathbf{X}_k\in\mathbb{R}^{2ks \times 2ks}, \quad
    \mathbf{R}_k\in\mathbb{R}^{(2k+1)s \times (2k+1)s}, \\
    \mathbf{P}_k\in\mathbb{R}^{(2k+1)s \times (2k+1)s}, \quad
    \mathbf{Q}_k\in\mathbb{R}^{(2k+2)s \times (2k+2)s}, \\
    \mathbf{S}_k\in\mathbb{R}^{(2k+2)s \times (2k+2)s}, \quad
    \mathbf{T}_k\in\mathbb{R}^{(2k+3)s \times (2k+3)s}.
\end{align*}
We use the block Arnoldi process to compute $\mathcal{V}_k$ and $\mathcal{W}_k$.

Using~\eqref{oper:addition}, the lines 8, 11, 12, and 18 of Algorithm~\ref{alg:BiCGSTAB} can be written as
\begin{align*}
    S_{k} &= R_k - \alpha_k Q_k = \mathcal{V}_{2k+2}\mathbf{S}_{k}\mathcal{W}_{2k+2}^\top, \\
    X_{k+1} &= X_k + \alpha_k P_k + \omega_k S_k = \mathcal{V}_{2k+2}\mathbf{X}_{k+1}\mathcal{W}_{2k+2}^\top, \\
    R_{k+1} &= S_k - \omega_k T_k = \mathcal{V}_{2k+3}\mathbf{R}_{k+1}\mathcal{W}_{2k+3}^\top, \\
    P_{k+1} &= R_{k+1} + \beta_k (P_k - \omega_k Q_k) = \mathcal{V}_{2k+3}\mathbf{P}_{k+1}\mathcal{W}_{2k+3}^\top,
\end{align*}
which implies the following recurrence formulas:
\begin{align*}
    \mathbf{S}_k    &:=\begin{bmatrix}\mathbf{R}_k & O_{(2k+1)s,s} \\ O_{s,(2k+1)s} & O_{s,s}\end{bmatrix} - \alpha_k\mathbf{Q}_{k}, \\ % \in\mathbb{R}^{(2k+2)s\times (2k+2)s}
    \mathbf{X}_{k+1}&:=\begin{bmatrix}\mathbf{X}_k & O_{2ks,2s} \\ O_{2s,2ks} & O_{2s,2s}\end{bmatrix} + \begin{bmatrix}\alpha_k\mathbf{P}_k & O_{(2k+1)s,s} \\ O_{s,(2k+1)s} & O_{s,s}\end{bmatrix} + \omega_k\mathbf{S}_k, \\ % \in\mathbb{R}^{(2k+2)s\times (2k+2)s}
    \mathbf{R}_{k+1}&:=\begin{bmatrix}\mathbf{S}_k & O_{(2k+2)s,s} \\ O_{s,(2k+2)s} & O_{s,s}\end{bmatrix} - \omega_k\mathbf{T}_k, \\ % \in\mathbb{R}^{(2k+3)s\times (2k+3)s}
    \mathbf{P}_{k+1}&:=\mathbf{R}_{k+1} + \beta_k\left(\begin{bmatrix}\mathbf{P}_k & O_{(2k+1)s,2s} \\ O_{2s,(2k+1)s} & O_{2s,2s}\end{bmatrix} - \omega_k \begin{bmatrix}\mathbf{Q}_k & O_{(2k+2)s,s} \\ O_{s,(2k+2)s} & O_{s,s}\end{bmatrix}\right). % \in\mathbb{R}^{(2k+3)s\times (2k+3)s}
\end{align*}

From~\eqref{oper:inner_product}, the inner products in the lines 7, 10, 16 of Algorithm~\ref{alg:BiCGSTAB} become
\begin{align*}
    \alpha_k &= \rho_k/\langle \tilde{\mathbf{R}}_0 , \mathbf{Q}_k(1:s,1:s) \rangle, \\
    \omega_k &= \langle \mathbf{T}_k(1:(2k+2)s,1:(2k+2)s) , \mathbf{S}_k \rangle/\langle \mathbf{T}_k , \mathbf{T}_k \rangle, \\
    \rho_{k+1} &= \langle \tilde{\mathbf{R}}_0 , \mathbf{R}_{k+1}(1:s,1:s) \rangle. 
\end{align*}

The calculation of the Sylvester operator in the lines 6 and 9 of Algorithm~\ref{alg:BiCGSTAB} becomes
\begin{align*}
    Q_{k} &= AP_{k} + P_{k}B = \mathcal{V}_{2k+2}\mathbf{Q}_{k}\mathcal{W}_{2k+2}^\top,\\
    T_{k} &= AS_{k} + S_{k}B = \mathcal{V}_{2k+3}\mathbf{T}_{k}\mathcal{W}_{2k+3}^\top,
\end{align*}
where
\begin{align*}
    \mathbf{Q}_{k}&:=\left[\mathcal{H}_{2k+2,2k+1}\mathbf{P}_k, O_{(2k+2)s,s}\right] + \begin{bmatrix}\mathbf{P}_k\mathcal{G}_{2k+2,2k+1}^\top \\ O_{s,(2k+2)s}\end{bmatrix}, \\ % \in\mathbb{R}^{(2k+2)s\times (2k+2)s}
    \mathbf{T}_{k}&:=\left[\mathcal{H}_{2k+3,2k+2}\mathbf{S}_k, O_{(2k+3)s,s}\right] + \begin{bmatrix}\mathbf{S}_k\mathcal{G}_{2k+3,2k+2}^\top \\ O_{s,(2k+3)s}\end{bmatrix},   % \in\mathbb{R}^{(2k+3)s\times (2k+3)s}
\end{align*}
by using~\eqref{oper:sylvester-operator}.

From the discussion, we have the {\it factorized BiCGSTAB method} presented in Algorithm~\ref{alg:f-BiCGSTAB}.
\begin{algorithm}
\caption{Factorized BiCGSTAB method for solving~\eqref{eq:sylvester}}
\label{alg:f-BiCGSTAB}
\begin{algorithmic}[1]
\REQUIRE $A\in\mathbb{R}^{n\times n},B\in\mathbb{R}^{m\times m},C_1\in\mathbb{R}^{n\times s}, C_2\in\mathbb{R}^{m\times s},\varepsilon_\mathrm{tol}$
\ENSURE $\mathcal{V}_{2k+2}\in\mathbb{R}^{n\times (2k+2)s},\mathbf{X}_{k+1}\in\mathbb{R}^{(2k+2)s\times (2k+2)s},\mathcal{W}_{2k+2}\in\mathbb{R}^{m\times (2k+2)s}$ such that $\mathcal{V}_{2k+2}\mathbf{X}_{k+1}\mathcal{W}_{2k+2}^\top$ is an approximate solution to~\eqref{eq:sylvester}
\STATE{$\mathbf{X}_0=\left[~\right]$}
\STATE{Compute QR factorizations: $C_1=V_1 R_\mathrm{L}, C_2=W_1R_\mathrm{R}$}
\STATE{$\mathbf{R}_0=R_\mathrm{L} R_\mathrm{R}^\top$, $\mathbf{P}_0=\mathbf{R}_0$, $\tilde{\mathbf{R}}_0=\mathbf{R}_0$, $\rho_0=\langle \tilde{\mathbf{R}}_0, \mathbf{R}_0\rangle$}
\FOR{$k=0,1,2,\ldots$}
\STATE{Compute the $(2k+1)$-th step of the block Arnoldi process for $\mathcal{K}_{2k+1}^\square(A,C_1)$ and obtain $\mathcal{V}_{2k+2}$ and $\mathcal{H}_{2k+2,2k+1}$}
\STATE{Compute the $(2k+1)$-th step of the block Arnoldi process for $\mathcal{K}_{2k+1}^\square(B^\top,C_2)$ and obtain $\mathcal{W}_{2k+2}$ and $\mathcal{G}_{2k+2,2k+1}$}
\STATE{$\mathbf{Q}_k = \left[\mathcal{H}_{2k+2,2k+1}\mathbf{P}_k, O_{(2k+2)s,s}\right] + \begin{bmatrix}\mathbf{P}_k\mathcal{G}_{2k+2,2k+1}^\top \\ O_{s,(2k+2)s}\end{bmatrix}$}
\STATE{$\alpha_k=\rho_k/\left\langle \tilde{\mathbf{R}}_0, \mathbf{Q}_k(1:s,1:s) \right\rangle$}
\STATE{$\mathbf{S}_k:=\begin{bmatrix}\mathbf{R}_k & O_{(2k+1)s,s} \\ O_{s,(2k+1)s} & O_{s,s}\end{bmatrix} - \alpha_k\mathbf{Q}_k$}
\STATE{Compute the $(2k+2)$-th step of the block Arnoldi process for $\mathcal{K}_{2k+2}^\square(A,C_1)$ and obtain $\mathcal{V}_{2k+3}$ and $\mathcal{H}_{2k+3,2k+2}$}
\STATE{Compute the $(2k+2)$-th step of the block Arnoldi process for $\mathcal{K}_{2k+2}^\square(B^\top,C_2)$ and obtain $\mathcal{W}_{2k+3}$ and $\mathcal{G}_{2k+3,2k+2}$}
\STATE{$\mathbf{T}_k = \left[\mathcal{H}_{2k+3,2k+2}\mathbf{S}_k, O_{(2k+3)s,s}\right] + \begin{bmatrix}\mathbf{S}_k\mathcal{G}_{2k+3,2k+2}^\top \\ O_{s, (2k+3)s}\end{bmatrix}$}
\STATE{$\omega_k = \langle \mathbf{T}_k(1:(2k+2)s,1:(2k+2)s) , \mathbf{S}_k \rangle/\langle \mathbf{T}_k , \mathbf{T}_k \rangle$}
\STATE{$\mathbf{X}_{k+1} = \begin{bmatrix}\mathbf{X}_k & O_{2ks,2s} \\ O_{2s,2ks} & O_{2s,2s}\end{bmatrix} + \begin{bmatrix}\alpha_k\mathbf{P}_k & O_{(2k+1)s,s} \\ O_{s,(2k+1)s} & O_{s,s}\end{bmatrix} + \omega_k\mathbf{S}_k$}
\STATE{$\mathbf{R}_{k+1} = \begin{bmatrix}\mathbf{S}_k & O_{(2k+2)s,s} \\ O_{s,(2k+2)s} & O_{s,s}\end{bmatrix} - \omega_k\mathbf{T}_k$}
\IF{$\|\mathbf{R}_{k+1}\|\le \|\mathbf{R}_0\|\cdot \varepsilon_\mathrm{tol}$}
\RETURN $\mathbf{X}_{k+1}$, $\mathcal{V}_{2k+2}$, and $\mathcal{W}_{2k+2}$
\ENDIF
\STATE{$\rho_{k+1} = \langle \tilde{\mathbf{R}}_0 , \mathbf{R}_{k+1}(1:s,1:s) \rangle$}
\STATE{$\beta_k=(\alpha_k/\omega_k)\cdot(\rho_{k+1}/\rho_k)$}
\STATE{$\mathbf{P}_{k+1}=\mathbf{R}_{k+1} + \beta_k\left(\begin{bmatrix}\mathbf{P}_k & O_{(2k+1)s,2s} \\ O_{2s,(2k+1)s} & O_{2s,2s}\end{bmatrix} - \omega_k \begin{bmatrix}\mathbf{Q}_k & O_{(2k+2)s,s} \\ O_{s,(2k+2)s} & O_{s,s}\end{bmatrix}\right)$}
\ENDFOR
\end{algorithmic}
\end{algorithm}

%% ================================================ %%
\subsection{The case of the Lyapunov equation}\label{subsec:Lyap}
When $B=A^\top$ and $C_2 = C_1$, the Sylvester equation~\eqref{eq:sylvester} becomes 
\begin{equation}
    AX+XA^\top  = C_1 C_1^\top,\label{eq:lyap}
\end{equation}
which is called the Lyapunov equation.
For this case, with the aid of the symmetry of~\eqref{eq:lyap}, the solution can be expressed as $X_k=\mathcal{V}_k\mathbf{X}_k\mathcal{V}_k^\top$ instead of~\eqref{form:VZW}.
This allows us to reduce the computation of the orthogonal basis by half.
Additionally, the number of matrix products appearing in the Sylvester operator (line 6 in Algorithm~\ref{alg:f-CG}) is reduced from two to one.
These modifications improve the computational efficiency of the method.
%% ================================================ %%

\section{Numerical experiments}
\label{sec:experiments}
In this section, we demonstrate the performance of the proposed methods through numerical experiments.
The benchmark problems are collected from earlier studies cited below.
All experiments were performed on an Intel Core i7-12700 CPU (2.10 GHz) with 32.0 GB RAM, running Windows 11 Pro, using MATLAB R2024a.

\subsection{Convergence behavior}
From the discussion in Section~\ref{sec:main}, the proposed methods are mathematically equivalent to the matrix-oriented Krylov subspace methods. 
Therefore, similar convergence behaviors can be expected.
We compare the factorized CG and BiCGSTAB methods with the matrix-oriented CG and BiCGSTAB methods on symmetric and nonsymmetric problems, respectively.
In this subsection, the threshold is set to $\varepsilon_\mathrm{tol}=10^{-8}$.

\begin{example}\label{ex1}
    We first consider solving the Lyapunov equation~\eqref{eq:lyap} with symmetric positive definite $A$ by the matrix-oriented CG method and the factorized CG method.
    The test matrix $A$ is taken from~\cite[Section 4.1]{kressner2021} with matrix size $n=10{,}000$, obtained by finite-difference discretization of the two-dimensional Laplacian operator, and $C_1\in\mathbb{R}^{n\times 3}$ is chosen randomly.
    It should be noted that in~\cite[Section 4.1]{kressner2021}, $n$ represents the number of grid points in each direction, while in this paper $n$ denotes the matrix size, i.e., the number of grid points is $\sqrt{n}=100$.

    Figure~\ref{fig:ex1_lyap} shows the convergence histories of the relative residual norms for the matrix-oriented CG and factorized CG methods.
    The result confirms that the two algorithms exhibit the same numerical behavior, which is consistent with their mathematical equivalence.
    \begin{figure}[htbp]
        \centering
        \includegraphics[width=0.7\linewidth]{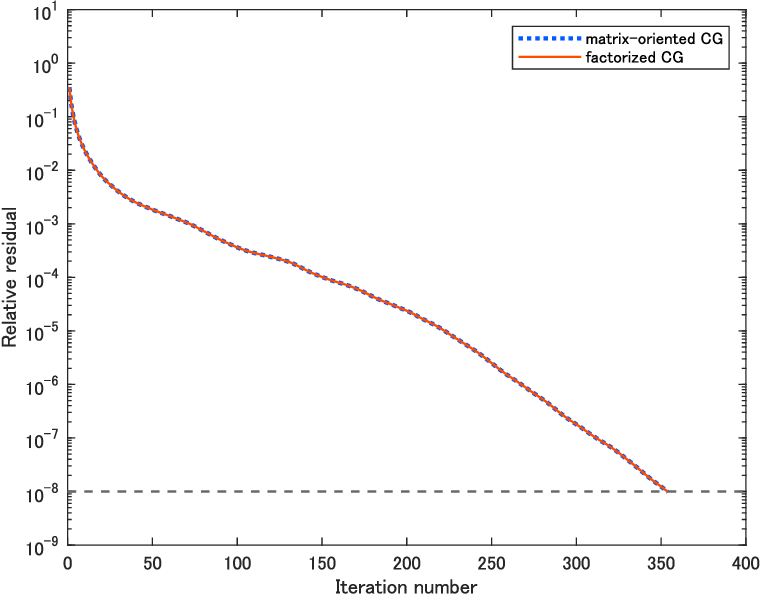}
        \caption{Convergence histories of the matrix-oriented CG method and the factorized CG method for Example~\ref{ex1}.}
        \label{fig:ex1_lyap}
    \end{figure}
\end{example}

\begin{example}\label{ex2}
    We consider solving the Sylvester equation~\eqref{eq:sylvester} with nonsymmetric $A$ and $B$ by the matrix-oriented BiCGSTAB method and the factorized BiCGSTAB method.
    The test matrices $A$ and $B$ are taken from~\cite[Section 4.2]{kressner2021} with the matrix size set to $m=n=8{,}000~(=20^3)$.
    The matrices $C_1,C_2\in\mathbb{R}^{n\times 3}$ are chosen randomly.

    Figure~\ref{fig:ex1_sylv} shows the convergence of the relative residual norms for the matrix-oriented BiCGSTAB and factorized BiCGSTAB methods.
    The two methods produce the same convergence history in the early iterations, whereas in the later iterations the residual norms oscillate and the curves no longer coincide. This is likely due to the different propagation of rounding errors in the two implementations. Even so, the two methods show comparable convergence behavior overall.
    \begin{figure}[htbp]
        \centering
        \includegraphics[width=0.7\linewidth]{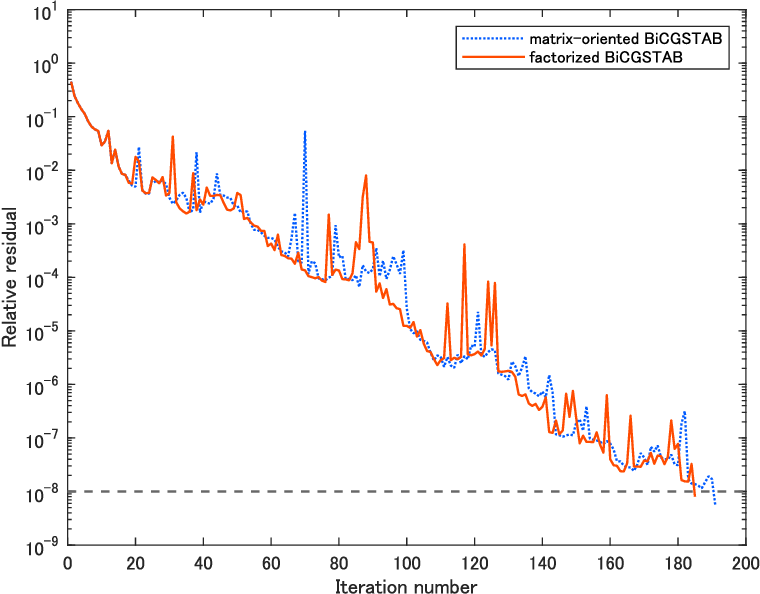}
        \caption{Convergence histories of the matrix-oriented BiCGSTAB method and the factorized BiCGSTAB method for Example~\ref{ex2}.}
        \label{fig:ex1_sylv}
    \end{figure}
\end{example}

\subsection{Comparison with the truncated Krylov subspace methods}\label{subsec:num-comp}
We next compare the factorized Krylov subspace methods with truncated Krylov subspace methods, which are the matrix-oriented Krylov subspace methods that employ low-rank truncation, in terms of computational time and the number of iterations for convergence.
The truncated Krylov subspace methods used for comparison are obtained by applying~\cite[Algorithms 1, 2]{benner2013} or~\cite[Algorithms 2, 3]{kressner2011} to the Sylvester or Lyapunov equation. 
The resulting truncated CG method for~\eqref{eq:lyap} and the truncated BiCGSTAB method for~\eqref{eq:sylvester}, used in Examples~\ref{ex3} and~\ref{ex4}, respectively, are detailed in Algorithms~\ref{alg:truncCG} and~\ref{alg:truncBiCGSTAB} in the appendix.
Note that no preconditioner is applied to these algorithms so that the truncated Krylov subspace methods are mathematically equivalent to the factorized Krylov subspace methods except for the low-rank truncation.
In our experiments, the truncation parameter is set to $\varepsilon_\mathcal{T}=10^{-8},10^{-10},10^{-12}$.
The threshold for convergence is set to $\varepsilon_\mathrm{tol}=10^{-6}$ in this subsection.
\begin{example}\label{ex3}
    The test matrix $A$ is the same as in Example~\ref{ex1}. The matrix $C_1\in\mathbb{R}^{n\times s}$ is chosen randomly, and we perform the experiments with $s=1,2,3,4,5$.
    As described in Subsection~\ref{subsec:Lyap}, the factorized CG method performed computations efficiently by exploiting the symmetry. 
    The truncated CG method also utilized this property for efficient computation.
    In this example, optional truncation of $R_k$ in the truncated CG method is enabled (line 9 of Algorithm~\ref{alg:truncCG}), while that of $Q_k$ is disabled (line 5 of Algorithm~\ref{alg:truncCG}).
    We chose this setting because we confirmed that this setting achieved the shortest or comparable computational time for each $s$ among all combinations of the optional truncations.
    \begin{figure}[htbp]
        \centering
        \includegraphics[width=0.7\linewidth]{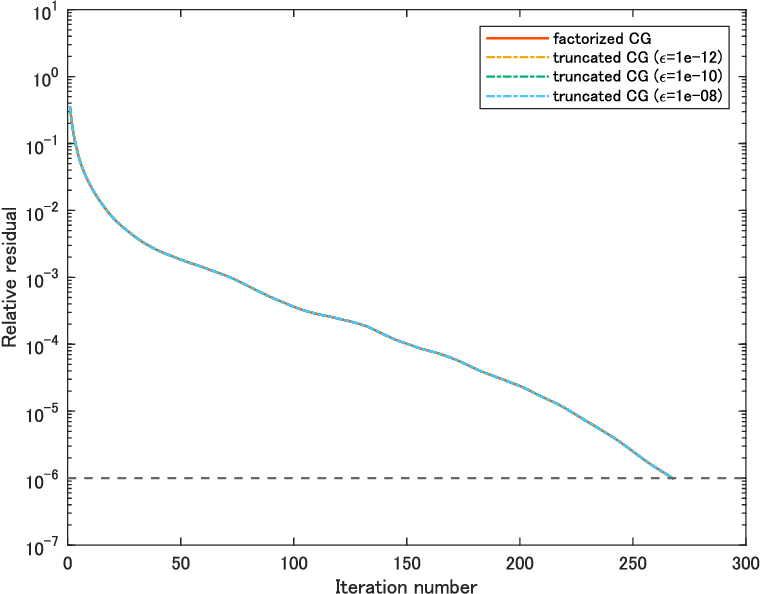}
        \caption{Convergence histories of the factorized CG method and the truncated CG methods ($\varepsilon_\mathcal{T}=10^{-8}, 10^{-10}, 10^{-12}$) for Example~\ref{ex3} with $s=3$.}
        \label{fig:ex3}
    \end{figure}

    Figure~\ref{fig:ex3} presents the convergence histories for the case $s=3$.
    The figure shows that the convergence behavior is identical among the factorized CG method and the truncated CG methods (with three different parameter settings).
    Similar behavior was observed for the other tested ranks, namely $s=1,2,4,5$.

    Table~\ref{tab:ex3} summarizes the iteration number and computational time of each method until convergence for each $s$.
    A breakdown of the computational time for each method is also provided.
    In each setting of $s$, while the iteration numbers are the same for all methods, the factorized CG method is faster than the truncated CG methods in terms of computational time. 
    This significant difference is primarily attributed to the fact that the factorized CG method does not require low-rank truncation that the truncated CG method must perform to avoid excessive memory consumption.
    An additional factor is that the basic operations in the factorized CG method reduce to small matrix computations, whereas the corresponding operations in the truncated CG method are carried out using a low-rank format, which is more expensive.
    For reference, the computational cost of each operation in both the factorized and truncated CG methods is summarized in Table~\ref{tab:cost-cg} in Appendix~\ref{appendix:B}.

    \begin{table}[htbp]
        \centering
        \small
        \setlength{\tabcolsep}{4pt}
        \begin{tabular}{llcccc}
        \toprule
        \multirow{2}{*}{$s$}
          & \multirow{3}{*}{} & \multirow{1}{*}{Factorized CG}
          & \multicolumn{3}{c}{Truncated CG} \\
         \cmidrule(lr){4-6}
          & & & {$\varepsilon_\mathcal{T}=10^{-8}$} & {$\varepsilon_\mathcal{T}=10^{-10}$} & {$\varepsilon_\mathcal{T}=10^{-12}$} \\
        \midrule
        
        % ----- s = 1 -----
        \multirow{5}{*}{1}
          & Number of iterations    & 270  & 270   & 270   & 270   \\
          & Total time              & 0.17 & 25.89 & 27.19 & 27.98 \\
          & ~ Basic operations      & 0.14 &  7.52 & 7.67  & 7.87  \\
          & ~ Block Lanczos process & 0.03 & -     & -     & -     \\
          & ~ Low-rank truncations  & -    & 18.37 & 19.51 & 20.10 \\
        \midrule
        
        % ----- s = 2 -----
        \multirow{5}{*}{2}
          & Number of iterations            & 270  & 270   & 270   & 270   \\
          & Total time               & 1.17 & 93.32 & 90.75 & 89.04 \\
          & ~ Basic operations      & 1.09 & 20.02 & 18.88 & 18.62 \\
          & ~ Block Lanczos process & 0.08 & -     & -     & -     \\
          & ~ Low-rank truncations  & -    & 73.28 & 71.86 & 70.41 \\
        \midrule
        
        % ----- s = 3 -----
        \multirow{5}{*}{3}
          & Number of iterations            & 268  & 268    & 268    & 268    \\
          & Total time              & 2.78 & 182.31 & 191.66 & 183.77 \\
          & ~ Basic operations      & 2.64 &  33.50 &  33.71 &  32.36 \\
          & ~ Block Lanczos process & 0.14 & -      & -      & -      \\
          & ~ Low-rank truncations  & -    & 148.79 & 157.94 & 151.36 \\
        \midrule
        
        % ----- s = 4 -----
        \multirow{5}{*}{4}
          & Number of iterations            & 268  & 268    & 268    & 268    \\
          & Total time              & 5.06 & 310.84 & 322.44 & 322.09 \\
          & ~ Basic operations      & 4.88 &  51.48 &  51.26 &  50.72 \\
          & ~ Block Lanczos process & 0.18 & -      & -      & -      \\
          & ~ Low-rank truncations  & -    & 259.35 & 271.17 & 271.36 \\
        \midrule
        
        % ----- s = 5 -----
        \multirow{5}{*}{5}
          & Number of iterations            & 271  & 271    & 271    & 271    \\
          & Total time              & 8.20 & 485.18 & 510.43 & 525.21 \\
          & ~ Basic operations      & 7.97 &  73.77 &  73.73 &  74.72 \\
          & ~ Block Lanczos process & 0.23 & -      & -      & -      \\
          & ~ Low-rank truncations  & -    & 411.39 & 436.68 & 450.47 \\
        \bottomrule
        \end{tabular}
        \label{tab:ex3}
        \caption{Number of iterations and breakdown of computational time (in seconds) of the factorized and truncated CG methods for Example~\ref{ex3}. ``Basic operations'' includes the time for the convergence check.}
    \end{table}
\end{example}

\begin{example}\label{ex4}
    We use the same matrices $A$ and $B$ as in Example~\ref{ex2}, while changing the matrix size to $n=15{,}625~(=25^3)$.
    The matrices $C_1,C_2\in\mathbb{R}^{n\times s}$ are chosen randomly, and we perform the experiments with $s=1,2,3,4,5$.
    For updating $R_k$ in the truncated BiCGSTAB method, while the existing studies have used the explicit formula (Variant 2 in Algorithm~\ref{alg:truncBiCGSTAB}) to prevent early stagnation of the residual, we adopt the recursion formula (Variant 1 in Algorithm~\ref{alg:truncBiCGSTAB}) because it converged in fewer iterations and resulted in shorter computational time.
    In this example, optional truncations of $Q_k$ and $S_k$ in the truncated BiCGSTAB method are enabled (lines 6 and 8 of Algorithm~\ref{alg:truncBiCGSTAB}), while that of $T_k$ is disabled (line 9 of Algorithm~\ref{alg:truncBiCGSTAB}). 
    This setting was chosen for the same reason as in Example~\ref{ex3}.

    Figure~\ref{fig:ex4} shows the convergence histories of the factorized BiCGSTAB method and the truncated BiCGSTAB methods (with three different parameter settings) for $s=3$. 
    For the first approximately 15 iterations, the behavior of all methods appears to be nearly identical, similar to Example~\ref{ex3}. 
    However, small differences gradually emerge among the methods, although the number of iterations until convergence remains comparable across all methods. 
    Similar trends were observed for the other tested ranks.
        
    Computational times and iteration numbers until convergence are summarized in Table~\ref{tab:ex4}.
    For each $s$, we confirmed that the factorized BiCGSTAB method is faster than the truncated BiCGSTAB methods in terms of computational time, despite a similar number of iterations.
    The reasons for the greater reduction in computational time relative to the reduction in the number of iterations are essentially the same as in Example~\ref{ex3}: the factorized BiCGSTAB method does not perform low-rank truncation, and its basic operations are more efficient than those in the truncated BiCGSTAB methods.
    However, the reduction in computational time is less pronounced than that in Example~\ref{ex3}, because the block Arnoldi process dominates the computational time due to its increasing cost per iteration, unlike the block Lanczos process used in the factorized CG method.
    As in the CG case, the computational cost of each operation in both the factorized and truncated BiCGSTAB methods is summarized in Table~\ref{tab:cost-bicgstab} in Appendix~\ref{appendix:B}.

    \begin{figure}[htbp]
        \centering
        \includegraphics[width=0.7\linewidth]{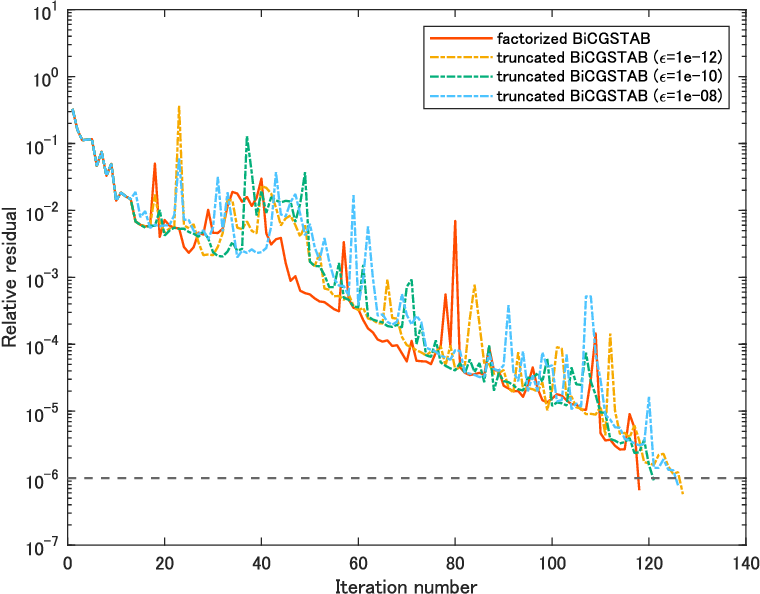}
        \caption{Convergence histories of the factorized BiCGSTAB method and the truncated BiCGSTAB methods ($\varepsilon_\mathcal{T}=10^{-8}, 10^{-10}, 10^{-12}$) for Example~\ref{ex4} with $s=3$.}
        \label{fig:ex4}
    \end{figure}

        \begin{table}[htbp]
        \centering
        \small  
        \setlength{\tabcolsep}{4pt}
        \begin{tabular}{llccccc}
        \toprule
        \multirow{2}{*}{$s$}
      & \multirow{2}{*}{} & \multirow{2}{*}{\makecell{Factorized \\ BiCGSTAB}}
          & \multicolumn{3}{c}{Truncated BiCGSTAB} \\
         \cmidrule(lr){4-6}
          & & & {$\varepsilon_\mathcal{T}=10^{-8}$} & {$\varepsilon_\mathcal{T}=10^{-10}$} & {$\varepsilon_\mathcal{T}=10^{-12}$} \\
        \midrule

        % ----- s = 1 -----
        \multirow{5}{*}{1}
          & Number of iterations            &  121 &   132 &   126 &   127 \\
          & Total time              & 2.07 & 11.80 & 12.48 & 14.76 \\
          & ~ Basic operations      & 0.22 &  3.38 &  3.57 &  4.01 \\
          & ~ Block Arnoldi process & 1.84 & -     & -     & -     \\
          & ~ Low-rank truncations  & -    &  8.41 &  8.90 & 10.74 \\
        \midrule
        
        % ----- s = 2 -----
        \multirow{5}{*}{2}
          & Number of iterations            &  112 &   124 &   120 &   115 \\
          & Total time              & 4.31 & 21.21 & 23.35 & 24.09 \\
          & ~ Basic operations      & 0.67 &  5.54 &  6.11 &  6.29 \\
          & ~ Block Arnoldi process & 3.64 & -     & -     & -     \\
          & ~ Low-rank truncations  & -    & 15.66 & 17.23 & 17.79 \\
        \midrule
        
        % ----- s = 3 -----
        \multirow{5}{*}{3}
          & Number of iterations            &  118 &   126 &   121 &   127 \\
          & Total time              & 9.03 & 36.21 & 41.41 & 46.55 \\
          & ~ Basic operations      & 1.91 &  8.74 &  9.67 & 10.49 \\
          & ~ Block Arnoldi process & 7.12 & -     & -     & -     \\
          & ~ Low-rank truncations  & -    & 27.45 & 31.72 & 36.05 \\
        \midrule
        
        % ----- s = 4 -----
        \multirow{5}{*}{4}
          & Number of iterations            &   116 &   117 &   128 &   123 \\
          & Total time              & 11.87 & 44.76 & 59.97 & 63.89 \\
          & ~ Basic operations      &  3.28 & 10.20 & 12.86 & 13.18 \\
          & ~ Block Arnoldi process &  8.58 & -     & -     & -     \\
          & ~ Low-rank truncations  & -     & 34.55 & 47.09 & 50.69 \\
        \midrule
        
        % ----- s = 5 -----
        \multirow{5}{*}{5}
          & Number of iterations            &   112 &   125 &   119 &   118 \\
          & Total time              & 15.26 & 72.29 & 82.04 & 97.86 \\
          & ~ Basic operations      &  5.18 & 14.81 & 15.93 & 17.63 \\
          & ~ Block Arnoldi process & 10.08 & -     & -     & -     \\
          & ~ Low-rank truncations  & -     & 57.46 & 66.09 & 80.21 \\
        \bottomrule
        \end{tabular}
        \label{tab:ex4}
        \caption{Number of iterations and breakdown of computational time (in seconds) of the factorized and truncated BiCGSTAB methods for Example~\ref{ex4}. ``Basic operations'' includes the time for the convergence check.}
    \end{table}
\end{example}

%% ================================================ %%
\section{Conclusions}\label{sec:conclusions}
In this paper, we have developed factorized Krylov subspace methods, which are low-rank versions of matrix-oriented Krylov subspace methods specifically designed for large-scale Sylvester equations with low-rank right-hand sides.
A key feature of the proposed approach is that we have shown the low-rank structure of the matrices arising in the matrix-oriented Krylov subspace methods, and based on this structure, we have reformulated the algorithms, thereby avoiding the explicit formation of large dense matrices without resorting to truncation.
Several numerical examples have demonstrated the computational efficiency of the proposed methods for large-scale Sylvester equations.

While this paper has focused on the CG and BiCGSTAB methods as the base methods for the factorized Krylov subspace methods, the proposed framework is expected to extend naturally to other Krylov subspace methods, such as the CGS and GPBiCG methods. 
Since it is difficult to apply a preconditioning technique to the proposed methods, the rank may grow substantially for ill-conditioned problems without achieving convergence within a small number of iterations. Developing improvements to handle such problems remains a topic for future work.

\section*{Acknowledgments}
This work was supported by JSPS KAKENHI Grant Number JP25K21213.

% --- Appendix --- %
\appendix
\section{The truncated CG and BiCGSTAB methods}\label{appendix:A}
The truncated CG method for the Lyapunov equation in Example~\ref{ex3} and the truncated BiCGSTAB method for the Sylvester equation in Example~\ref{ex4} are shown in Algorithm~\ref{alg:truncCG} and Algorithm~\ref{alg:truncBiCGSTAB}, respectively.
The symbol $\mathcal{T}$ denotes the truncation operator, which is defined in~\cite[p.5]{kressner2011}.
Note that the approximate solution $X_{k}$ and matrices $R_k$, $P_k$, $Q_k$, $S_k$, $T_k$ are stored in low-rank form and all operations involving low-rank matrices (addition, inner product, and the Sylvester operator) are performed efficiently by leveraging the low-rank structure.
Moreover, in Algorithm~\ref{alg:truncCG}, we also exploit the symmetry of $X_k$, $R_k$, $P_k$, and $Q_k$ to reduce the computational cost and memory consumption, as described in~\cite{kressner2014, simoncini2023}.
In our numerical experiments, the low-rank truncation $\mathcal{T}$ is also computed efficiently by exploiting symmetry.

\begin{algorithm}
\caption{Truncated CG method for solving~\eqref{eq:lyap}}
\label{alg:truncCG}
\begin{algorithmic}[1]
\REQUIRE $A\in\mathbb{R}^{n\times n},C_1\in\mathbb{R}^{n\times s},\varepsilon_\mathrm{tol}$, $\varepsilon_\mathcal{T}$
\ENSURE $X_{k+1}$ in low-rank format 
\STATE{Set initial guess $X_0$}
\STATE{Compute $R_0=C_1C_1^\top-AX_0-X_0 A^\top$}
\STATE{$P_0=R_0$}
\FOR{$k=0,1,2,\ldots$}
\STATE{$Q_{k} = AP_{k}+P_{k}A^\top$, \hfill Optionally: $Q_{k}\leftarrow\mathcal{T}(Q_{k})$}
\STATE{$\xi_{k}=\langle P_{k}, Q_{k}\rangle$}
\STATE{$\alpha_k=\langle R_k, P_k \rangle/\xi_k$}
\STATE{$X_{k+1}=X_k + \alpha_k P_k$, \hfill $X_{k+1}\leftarrow\mathcal{T}(X_{k+1})$}
% \STATE{$R_{k+1}=R_k - \alpha_k Q_k$, \hfill Optionally: $R_{k+1}\leftarrow\mathcal{T}(R_{k+1})$}
\STATE{$R_{k+1}=C_1C_1^\top - AX_{k+1} - X_{k+1}A^\top$, \hfill Optionally: $R_{k+1}\leftarrow\mathcal{T}(R_{k+1})$}
\IF{$\|R_{k+1}\|\le \|R_0\|\cdot \varepsilon_\mathrm{tol}$}
\RETURN $X_{k+1}$
\ENDIF
\STATE{$\beta_k=-\langle R_{k+1}, Q_k\rangle/\xi_k$}
\STATE{$P_{k+1}=R_{k+1} + \beta_k P_k$, \hfill $P_{k+1}\leftarrow\mathcal{T}(P_{k+1})$}
\ENDFOR
\end{algorithmic}
\end{algorithm}

\begin{algorithm}
\caption{Truncated BiCGSTAB method for solving~\eqref{eq:sylvester}}
\label{alg:truncBiCGSTAB}
\begin{algorithmic}[1]
\REQUIRE $A\in\mathbb{R}^{n\times n},B\in\mathbb{R}^{m\times m},C_1\in\mathbb{R}^{n\times s}, C_2\in\mathbb{R}^{m\times s},\varepsilon_\mathrm{tol}$, $\varepsilon_\mathcal{T}$
\ENSURE $X_{k+1}\in\mathbb{R}^{n\times m}$ in low-rank format
\STATE{Set initial guess $X_0$}
\STATE{Compute $R_0=C_1C_2^\top-AX_0-X_0 B$}
\STATE{Set an arbitrary matrix $\tilde{R}_0$ s.t. $\langle R_0,\tilde{R}_0\rangle\ne 0$, e.g., $\tilde{R}_0=R_0$}
\STATE{$P_0=R_0$, $\rho_0=\langle R_0, \tilde{R}_0\rangle$}
\FOR{$k=0,1,2,\ldots$}
\STATE{$Q_{k} = AP_{k}+P_{k}B$, \hfill Optionally: $Q_{k}\leftarrow\mathcal{T}(Q_{k})$}
\STATE{$\alpha_k=\rho_k/\langle \tilde{R}_0, Q_k \rangle$}
\STATE{$S_k=R_k - \alpha_kQ_k$, \hfill Optionally: $S_{k}\leftarrow\mathcal{T}(S_{k})$}
\STATE{$T_{k} = AS_{k}+S_{k}B$, \hfill Optionally: $T_{k}\leftarrow\mathcal{T}(T_{k})$}
\STATE{$\omega_k = \langle T_k,S_k\rangle/\langle T_k,T_k\rangle$}
\STATE{$X_{k+1}=X_k + \alpha_k P_k + \omega_k S_k$, \hfill $X_{k+1}\leftarrow\mathcal{T}(X_{k+1})$}
% \STATE{$R_{k+1}=S_k - \omega_k T_k$ \hfill  $R_{k+1}\leftarrow\mathcal{T}(R_{k+1})$}
\STATE{Variant 1: $R_{k+1}=S_k - \omega_k T_k$, \hfill  $R_{k+1}\leftarrow\mathcal{T}(R_{k+1})$}
\STATE{Variant 2: $R_{k+1}=C_1C_2^\top - AX_{k+1} - X_{k+1}B$, \hfill Optionally: $R_{k+1}\leftarrow\mathcal{T}(R_{k+1})$}
\IF{$\|R_{k+1}\|\le \|R_0\|\cdot \varepsilon_\mathrm{tol}$}
\RETURN $X_{k+1}$
\ENDIF
\STATE{$\rho_{k+1}=\langle \tilde{R}_0, R_{k+1}\rangle$}
\STATE{$\beta_k=\frac{\alpha_k}{\omega_k}\cdot\frac{\rho_{k+1}}{\rho_k}$}
\STATE{$P_{k+1}=R_{k+1} + \beta_k (P_k - \omega_k Q_k)$, \hfill $P_{k+1}\leftarrow\mathcal{T}(P_{k+1})$}
\ENDFOR
\end{algorithmic}
\end{algorithm}

\section{Comparison of computational costs between the factorized and truncated Krylov subspace methods}\label{appendix:B}
Tables~\ref{tab:cost-cg} and~\ref{tab:cost-bicgstab} summarize the computational cost of each operation on low-rank matrices in the CG method for~\eqref{eq:lyap} and the BiCGSTAB method for~\eqref{eq:sylvester}, respectively.
For simplicity, we assume that the low-rank matrices have size $n \times n$ and rank $r$, where rank refers to the number of columns of the tall-skinny factors in the low-rank format.
In Table~\ref{tab:cost-bicgstab}, $k$ denotes the step index of the block Arnoldi process, not the iteration index of BiCGSTAB method.
The cost of the addition in the truncated Krylov subspace methods is zero because the operation is performed by stacking tall-skinny matrices horizontally.
\begin{table}[htbp]
  \centering
  \label{tab:cost-cg}
  \begin{tabular}{lcc}
    \toprule
    Operation & Factorized CG & Truncated CG \\
    \midrule
    Addition              & $\mathcal{O}(r^2)$  & 0                          \\
    Inner product         & $\mathcal{O}(r^2)$  & $\mathcal{O}(nr^2)$        \\
    Lyapunov operator    & $\mathcal{O}(r^3)$  & $\mathcal{O}(\mathrm{nnz}(A)\cdot r)$  \\
    \midrule
    Low-rank truncation   & ---                                           & $\mathcal{O}(nr^2 + r^3)$ \\
    Block Lanczos process & $\mathcal{O}(\mathrm{nnz}(A)s + ns^2)$  & ---         \\
    \bottomrule
  \end{tabular}
  \caption{Computational cost of each operation in the factorized and truncated CG methods for~\eqref{eq:lyap}.}
\end{table}
\begin{table}[htbp]
  \centering
  \label{tab:cost-bicgstab}
  \begin{tabular}{lcc}
    \toprule
    Operation & Factorized BiCGSTAB & Truncated BiCGSTAB  \\
    \midrule
    Addition              & $\mathcal{O}(r^2)$  & 0                     \\
    Inner product         & $\mathcal{O}(r^2)$  & $\mathcal{O}(nr^2)$  \\
    Sylvester operator    & $\mathcal{O}(r^3)$  & $\mathcal{O}((\mathrm{nnz}(A)+\mathrm{nnz}(B))r)$ \\
    \midrule
    Low-rank truncation   & ---  & $\mathcal{O}(nr^2 + r^3)$   \\
    Block Arnoldi process & $\mathcal{O}((\mathrm{nnz}(A)+\mathrm{nnz}(B))s + kns^2)$  & ---    \\
    \bottomrule
  \end{tabular}
  \caption{Computational cost of each operation in the factorized and truncated BiCGSTAB methods for~\eqref{eq:sylvester}.}
\end{table}

Note that the rank $r$ differs across matrices and also differs between the factorized and truncated Krylov subspace methods, so the tables do not allow a straightforward comparison of these methods.
However, for comparable ranks, each operation in the factorized Krylov subspace methods is less expensive than the inner product or the low-rank truncation in the truncated Krylov subspace methods, which is consistent with the numerical results in Subsection~\ref{subsec:num-comp}.

\bibliographystyle{siam}
\bibliography{references}
\end{document}